\documentclass[reqno,12pt]{amsart}
\usepackage{geometry}
\usepackage{amsmath,amssymb,mathrsfs,color,stackengine,mathtools}
\usepackage[colorlinks,
linkcolor=red,
anchorcolor=green,
citecolor=blue,
]{hyperref}
\usepackage[upint]{stix}
\usepackage{subfigure}
\usepackage{float}
\usepackage{times}
\usepackage{tikz}
\usepackage{rotating}
\usetikzlibrary{intersections}
\usepackage{paralist}
\usepackage{comment}

\makeatletter
\@namedef{subjclassname@2020}{\textup{2020} Mathematics Subject Classification}
\makeatother

\makeatletter
\def\setliststart#1{\setcounter{\@listctr}{#1}%
	\addtocounter{\@listctr}{-1}}
\makeatother

\makeatletter
\@addtoreset{figure}{section}
\makeatother

\usepackage{calc}
\newtheorem{The}{Theorem}[section]
\newtheorem{Cor}[The]{Corollary}
\newtheorem{Lem}[The]{Lemma}
\newtheorem{Pro}[The]{Proposition}
\theoremstyle{definition}
\newtheorem{defn}[The]{Definition}
\newtheorem{Rem}[The]{Remark}

\numberwithin{equation}{section}

\newcounter{Mr}
\newtheorem{Result}[Mr]{\textbf{Main Result}}

\newcommand{\R}{\mathbb{R}}

\newcommand{\N}{\mathbb{N}}

\newcommand{\SING}{\mbox{\rm Sing}\,}

\title[Global propagation of singularities for magnetic mechanical systems]{Global propagation of singularities for magnetic mechanical systems}
\author{Piermarco Cannarsa, Wei Cheng, Jiahui Hong \and Wenxue Wei\textsuperscript{*}}
\address[Piermarco Cannarsa]{Dipartimento di Matematica, Universit\`a di Roma ``Tor Vergata'', Via della Ricerca Scientifica 1, 00133 Roma, Italy}
\email{cannarsa@mat.uniroma2.it}
\address[Wei Cheng]{School of Mathematics, Nanjing University, Nanjing 210093, China}
\email{chengwei@nju.edu.cn}
\address[Jiahui Hong]{School of Mathematics, Nanjing University of Aeronautics and Astronautics, Nanjing 211106, China}
\email{hongjiahui@nuaa.edu.cn}
\address[Wenxue Wei]{School of Mathematics, Nanjing University, Nanjing 210093, China}
\email{wwx3708@gmail.com}
\begin{document}
	
	\maketitle
	
	\begin{abstract}
		We prove that singularities propagate globally for viscosity solutions of Hamilton-Jacobi equations related to magnetic mechanical systems on closed Riemannian manifolds. Our main result shows that for any weak KAM solution $u$, the singular set $\text{Sing}\,(u)$ remains invariant under the generalized gradient flow dynamics. The proof combines three key elements: (1) reduction from magnetic to Riemannian systems, (2) analysis of reparameterized flows, and (3) regularization techniques. Compared to previous analytic approaches, our geometric method provides clearer insights into the underlying Riemannian structure. We also establish necessary conditions for singularity existence, particularly when the Euler characteristic is nonzero and the magnetic form is non-exact. This approach does not extend directly to Finsler metrics due to structural differences.

	\end{abstract}
	
	\section{Introduction}
	
	This paper is devoted to studying the global propagation of singularities for viscosity solutions of the Eikonal equation on standard Riemannian manifolds and extended magnetic Lagrangian systems on closed (compact without boundary) manifolds.
	
	Let \((M, g)\) be a connected complete Riemannian manifold and \(u \colon W \to \mathbb{R}\) be a viscosity solution to the Eikonal equation
	\begin{equation}\label{eq:Eikonal_intro}
		\|\nabla u(x)\|_g^2 = 1, \quad x \in W,
	\end{equation}
	where \(W \subset M\) is an open subset, and \(\|\cdot\|_g\) denotes the norm induced by the Riemannian metric \(g\). Let \(d(\cdot, \cdot)\) be the Riemannian distance associated with \(g\). The distance function to the closed set \(F = M \setminus W\), defined by
	\[
	d_F(x) = \inf\,\{\, d(x, y) : y \in F \,\}, \quad x \in M,
	\]
	is well known to yield a viscosity solution of \eqref{eq:Eikonal_intro} on \(W\).
	
	The first global propagation of singularities for the distance function was established in \cite{ACNS2013} for both Euclidean and Riemannian settings (see also \cite{Cannarsa_Mazzola_Sinestrari2015} for the viscosity solutions to a special type of evolutionary mechanical Hamilton-Jacobi equations). Recently, global propagation results were also obtained in \cite{ACCM2025} for weak KAM solutions\footnote{We call a function $u:M\to \R$ a weak KAM solution of \eqref{eq:HJ_intro} if it is a viscosity solution of \eqref{eq:HJ_intro}.} of the standard mechanical system 
	\[
	H(x,p) = \tfrac{1}{2}|p|^2 + V(x)
	\] 
	on the torus. However, these prior works provide limited geometric insight. We develop a purely geometric approach to this problem, which extends to more general magnetic systems.
	
	Let \(\omega\) be a $C^2$ 1-form (not necessarily closed) on \(M\). The magnetic Lagrangian \(L: TM \to \mathbb{R}\) is defined by
	\begin{equation}\label{eq:magnetic_L_intro}
		L(x,v) = \tfrac{1}{2} g_x(v,v) - \omega_x(v) - V(x),
	\end{equation}
	where \(V \in C^2(M)\) is a potential function. By adding a constant to \(V\), we assume without loss of generality that \(\max_{x\in M} V(x) = 0\), as constant shifts preserve the Euler-Lagrange equations.
	
	The corresponding magnetic Hamiltonian \(H: T^*M \to \mathbb{R}\) is defined as 
	\begin{equation}\label{eq:magnetic_H_intro}
		H(x,p) =\sup_{v \in T_x M} \bigl\{ \langle p, v \rangle - L(x, v) \bigr\}= \tfrac{1}{2} g_x^*(p + \omega_x, p + \omega_x) + V(x),
	\end{equation}
	where \(g^*\) denotes the dual metric on \(T^*M\). The Hamilton-Jacobi equation is:
	\begin{equation}\label{eq:HJ_intro}
		H(x, Du) = \tfrac{1}{2} \|Du + \omega\|_{g^*}^2 + V(x) = c,
	\end{equation}
	where \(c = c(L)\) is the Ma\~n\'e's critical value and \(\|\cdot\|_{g^*}\) is the dual norm induced by \(g\). The vector field \(X = \omega^\sharp\) satisfies \(\omega = g(X, \cdot)\) via musical isomorphism. The identity \(\|Du + \omega\|_{g^*}^2 = \langle \nabla u + X, \nabla u + X \rangle\) holds.
	
	To state our main result, we introduce two subsets associated with the singular set of a solution \( u \) to \eqref{eq:HJ_intro}. 
	Denote by \(\SING(u)\) the set of points where \( u \) is not differentiable. 
	A point \( x \in M \) is called a cut point of \( u \) if \( \tau_u(x) = 0 \), where
	\begin{align*}
		\tau_u(x) = \sup \biggl\{ t \geq 0 : 
		\begin{array}{c}
			\exists \text{ a } C^1 \text{ curve } \gamma \colon [0,t] \to M, \gamma(0)=x \\
			\text{such that } u(\gamma(t)) - u(x) = \displaystyle\int_0^t L(\gamma, \dot{\gamma})  \mathrm{d}s
		\end{array}
		\biggr\}.
	\end{align*} 
	It is well known that \(\SING(u) \subset \mathrm{Cut}(u)\) holds in general.
	
	\begin{Result}
		Let \(u\) be a weak KAM solution to \eqref{eq:HJ_intro} with $H$ a magnetic Hamiltonian. Consider the generalized gradient flow
		\begin{equation}\label{eq:ggf_intro}
			\dot{\gamma}(t) = H_p\bigl( \gamma(t), p_u^\#(\gamma(t)) \bigr),\,  t\in[0,+\infty)\, a.e., \quad \gamma(0) = x_0,
		\end{equation}
		where $\gamma$ is a  Lipschitz curve and the momentum selection \(p_u^\#(x)\) is defined as 
		\[
		p_u^\#(x) := \underset{p \in D^+u(x)}{\arg\min}\, H(x,p)
		\]
		with \(D^+u(x)\) the superdifferential of \(u\) at \(x\). Then for any initial point \(x_0 \in W \cap\text{Cut}\,(u)\) with \(W=\{x \in M \mid c[L] > V(x)\}\), the solution \(\gamma\) satisfies
		\[
		\gamma(t) \in \SING(u) \quad \forall t \geq 0 \text{ in the flow's domain}.
		\]
		Thus \(\SING(u)\) is invariant under the generalized gradient flow dynamics. (See Theorem \ref{thm:singular_persistence} and Corollary \ref{cor:cut to sing}.)
	\end{Result}
	
	The core steps of the proof are as follows:
	\begin{enumerate}[--]
		\item Transform the magnetic system into an equivalent Eikonal-type problem via a Riemannian metric $\tilde{g}$ which is singular on the projected Aubry set - an essential construct in weak KAM theory (cf. \cite{Fathi_book}). 
		\item This geometric transformation reduces the magnetic case to a purely Riemannian framework. The generalized gradient flow \eqref{eq:ggf_intro} undergoes only a smooth time-reparametrization.
		\item Implement the regularization technique pioneered in \cite{Cannarsa_Yu2009, Yu2006} to provide a qualitative description of the energy evolution along the generalized gradient flow through careful computation of some geometric tensors.
	\end{enumerate}
	
	To elucidate our approach, we present a detailed proof for the Riemannian case in Section~\ref{sec:Riemannian}. In contrast to prior work \cite{ACCM2025}, our proof reveals deeper geometric structure even in this fundamental setting, uncovering the intrinsic Riemannian geometry that governs singularity propagation. However,  our method does not directly extend to the general Finslerian setting due to the lack of Riemannian structure (see Remark \ref{rem:Finsler}). This presents a significant difficulty when considering generalizations to Tonelli Hamiltonian systems. 
	
	For general Tonelli Hamiltonian systems, the theory of \emph{intrinsic singular characteristics} developed in \cite{Cannarsa_Cheng3,Cannarsa_Cheng_Fathi2017,Cannarsa_Cheng_Fathi2021} yields a global propagation result for singularities. These characteristics differ fundamentally from generalized gradient flow solutions of \eqref{eq:ggf_intro}, as demonstrated in \cite{CCHW2024}. The key achievement in \cite{CCHW2024} establishes that for any weak KAM solution $u$: every solution $\gamma:[0,+\infty) \to M$ of \eqref{eq:ggf_intro} starting at $x_0 \in \mathrm{Cut}\,(u)$, the cut locus of $u$ which is a broader class than $\text{Sing}\,(u)$, satisfies $\gamma(t) \in \mathrm{Cut}\,(u)$ for all $t \geq 0$. This statement is true even when $\gamma$ is a \emph{generalized characteristic} firstly introduced in \cite{Albano_Cannarsa2002} (see also \cite{Khanin_Sobolevski2016} for related topics).
	
	For any Tonelli Hamiltonian $H \in C^2(T^*M)$ and weak KAM solution $u$ of
	\begin{align*}
		H(x, Du) = c \quad \text{(Ma\~n\'e's critical value)},
	\end{align*}
	one has the following conjecture:
	\begin{center}
		\fbox{\begin{minipage}{0.9\linewidth}
				\vspace{2mm}
				\textit{If $\gamma:[0,\infty) \to M$ is a generalized gradient flow solution for \eqref{eq:ggf_intro} with $\gamma(0) \in \mathrm{Cut}\,(u)$, does $\gamma(t) \in \mathrm{Sing}\,(u)$ hold for all $t \geq 0$?}
				\vspace{2mm}
		\end{minipage}}
	\end{center}
	This persistence of singular points is a stronger property than cut locus propagation and remains an important open problem in general. 
	
	We also establish a criterion for the existence of singularities for magnetic  with Lagrangian systems under specific topological constraints (Theorem \ref{thm:sing_exists}).
	
	\begin{Result}
		Let $(M,g)$ be a connected closed orientable surface with Euler characteristic $\chi(M) \neq 0$. Consider the Lagrangian 
		\[
		L(x,v) = \tfrac{1}{2}g_x(v,v) - \omega_x(v) - V(x)
		\]
		where $\max_{x\in M} V(x) = 0$, and let $u$ be a weak KAM solution of \eqref{eq:HJ_intro}. If $\SING(u) = \emptyset$, then $c[L] = \max V = 0$. In particular, $\SING(u) \neq \emptyset$ when $V \equiv 0$ and $\omega$ is non-exact.
	\end{Result}
	
	This paper is organized as follows. In Section~\ref{sec:pre}, we recall some basic facts about semiconcave functions on manifolds. In Section~\ref{sec:gps}, we focus on the study of the properties of global propagation of singularities on Riemannian manifolds, for Eikonal equation in subsection \ref{sec:Riemannian} and Hamilton-Jacobi equation with magnetic Hamiltonian in subsection \ref{sec:magnetic}.
	
	\medskip
	
	\noindent\textbf{Acknowledgements.} Piermarco Cannarsa was supported, in part, by the National Group for Mathematical Analysis, Probability and Applications (GNAMPA) of the Italian Istituto Nazionale di Alta Matematica ``Francesco Severi'', by the Excellence Department Project awarded to the Department of Mathematics, University of Rome Tor Vergata, CUP E83C23000330006, and by the European Union---Next Generation EU, CUP E53D23017910001. Wei Cheng is partly supported by National Natural Science Foundation of China (Grant No. 12231010). Jiahui Hong is partly supported by National Natural Science Foundation of China (Grant No. 12501245).
	
	\section{Preliminaries}\label{sec:pre}
	
	This section collects fundamental notions on semiconcave and semiconvex functions on manifolds. Standard references for this topic include the monographs \cite{Ambrosio_GigliNicola_Savare_book2008, Cannarsa_Sinestrari_book, Villani_book2009}.
	
	\begin{defn}
		Let \( U \subset \mathbb{R}^n \) be an open convex set. A function \( f: U \to \mathbb{R} \) is said to be \emph{\( K \)-semiconcave} on \( U \) if, for every \( x, y \in U \) and every \( t \in [0,1] \), the following inequality holds
		\[
		(1-t)f(x) + tf(y) - f((1-t)x + ty) \leq K \frac{t(1-t)}{2} |x - y|^2,
		\]
		where \( K \geq 0 \) is a constant. A function \( f \colon U \to \mathbb{R} \) is called \emph{\( K \)-semiconvex} on \( U \) if \( -f \) is \( K \)-semiconcave on \( U \). 
		
		We say that a function \( f \colon U \to \mathbb{R} \), defined on an open subset of \( \mathbb{R}^n \), is \emph{locally semiconcave} (respectively, \emph{locally semiconvex}) if for every point \( x \in U \), there exists an open convex neighborhood \( U_x \subset U \) of \( x \) such that the restriction \( f|_{U_x} \) is \( K(x) \)-semiconcave (respectively, \( K(x) \)-semiconvex) for some constant \( K(x) \geq 0 \).
	\end{defn}
	
	\begin{defn}
		A function \( f \colon M \to \mathbb{R} \) defined on a manifold \( M \) is called \emph{locally semiconcave} (respectively, \emph{locally semiconvex}) if for every \( x \in M \), there exists a local chart \( \phi: U \to \mathbb{R}^n \) with \( x \in U \), such that the composition \( f \circ \phi^{-1} \colon \phi(U) \to \mathbb{R} \) is locally semiconcave (respectively, locally semiconvex) on \( \phi(U) \).
	\end{defn}
	
	\begin{defn}
		Let \((M, g)\) be a geodesically complete Riemannian manifold and \(U \subset M\) an open subset. A function \(f \colon U \to \mathbb{R}\) is called \emph{geodesically \(K\)-semiconcave} on \(U\) if for every minimal geodesic \(\gamma \colon [0, 1] \to U\) and every \(t \in [0, 1]\), the following inequality holds
		\begin{equation}\label{defn of sc}
			(1-t)f(\gamma(0)) + t f(\gamma(1)) - f(\gamma(t)) \leq \frac{K}{2} t(1-t) d_g^2(\gamma(0), \gamma(1)),
		\end{equation}
		where \(K \geq 0\) is a constant and \(d_g\) denotes the Riemannian distance induced by \(g\).
	\end{defn}
	
	\begin{defn}\label{def:local_geodesic_semiconcavity}
		Let \((M, g)\) be a connected, geodesically complete Riemannian manifold and \(U \subset M\) an open subset. A function \(f \colon U \to \mathbb{R}\) is called \emph{locally geodesically semiconcave} if for every \(x \in U\), there exists
		\begin{enumerate}
			\item a geodesically convex neighborhood \(U_x \subset U\) of \(x\) (meaning any minimal geodesic segment with endpoints in \(U_x\) lies entirely in \(U_x\)),
			\item a constant \(K(x) \geq 0\),
		\end{enumerate}
		such that \(f\) satisfies the geodesic \(K(x)\)-semiconcavity condition on \(U_x\).
	\end{defn}
	
	Applying the approximation theory developed by Greene and Wu \cite{Greene_Wu1972,Greene_Wu1974}, the following result was established in \cite[Theorem 6.6]{Ohta2009} and independently in \cite{Bangert1979}.
	
	\begin{The}\label{thm:semiconcave_equivalence}
		Let $(M,F)$ be a connected geodesically complete Riemannian manifold and $f \colon M \to \mathbb{R}$ a continuous function. Then $f$ is locally semiconcave if and only if $f$ is locally geodesically semiconcave.
	\end{The}
	\begin{defn}
		Let \( f: M \rightarrow \mathbb{R} \) be a function. We say that \( p \in T^{\ast}_{x}M \) is a superdifferential of \( f \) at \( x \in M \), and we write \( p \in D^{+}f(x) \), if there exists a function \( g: V \rightarrow \mathbb{R} \), defined on some open subset \( U \subset M \) containing \( x \), such that \( g \geq f \), and \( g \) is differentiable at \( x \) with \( D g(x) = p \). 
		
		Similarly, we say that \( p \in T^{\ast}_{x}M \) is a subdifferential of \( f \) at \( x \in M \), and we write \( p \in D^{-}f(x) \), if there exists a function \( g: V \rightarrow \mathbb{R} \), defined on some open subset \( U \subset M \) containing \( x \), such that \( g \leq f \), and \( g \) is differentiable at \( x \) with \( D g(x) = p \).
	\end{defn}
	
\begin{Rem}\label{basic property of sc}
\hfill
		\begin{enumerate}[\rm (1)]
			\item Suppose \( f \colon M \to \mathbb{R} \) is locally semiconcave. Then \( D^{+}f(x) \neq \varnothing \) for all \( x \in M \), and the directional derivative satisfies $\partial f(v_x) = \min\{ p(v_x): p \in D^{+}f(x)\}$. We call $p\in D^*f(x)$, the set of reachable differentials, if there exists a sequence $x_n\to x$ as $k\to +\infty$ with $f$ is differentiable at each $x_n$ and $p=\lim\limits_{n\to +\infty} Df(x_n)$. Furthermore, we have $D^*f(x)\subset D^+f(x)$ and $\operatorname{co}D^*f(x)= D^+f(x)$.
			\item Let \( (M,g) \) be a Riemannian manifold and \( f: M \rightarrow \mathbb{R} \) be a function. The superdifferential \( D^+f(x) \) (resp. subdifferential \( D^-f(x) \)) corresponds to the supergradient \( \nabla^+f(x) \) (resp. subgradient \( \nabla^-f(x) \)) with respect to the Riemannian inner product.
			\item For the magnetic Lagrangian \( L(x,v) = \tfrac{1}{2} g_x(v,v) - \omega_x(v) - V(x) \), the generalized gradient flow of \eqref{eq:ggf_intro} is given by
			\begin{align*}
				\dot{\gamma}(t) = H_p\bigl( \gamma(t), p^\#_u(\gamma(t)) \bigr) = \nabla^\# u(\gamma(t)) + X(\gamma(t)), \quad \gamma(0) = x_0,
			\end{align*}
			where \( \nabla^\# u = (p^\#_u)^\sharp \) and \(X=\omega^\sharp\).
			\item Suppose \( f \colon U \to \mathbb{R} \) is geodesically \( K \)-semiconcave on \( U \). Then for each \( v \in \nabla^+f(x) \) and every minimizing geodesic \( \gamma \colon [0, 1] \to U \) connecting \( x=\gamma(0) \) and \( \gamma(1) \), we have
			\begin{equation}\label{variation of sc}
				f(\gamma(1)) - f(\gamma(0)) \leq \langle v, \dot{\gamma}(0) \rangle + \frac{K}{2} d_g^2(\gamma(0), \gamma(1)).
			\end{equation}
			To see this, rearrange \eqref{defn of sc} to obtain
			\[
			f(\gamma(1)) - f(\gamma(0)) \leq \frac{f(\gamma(t)) - f(\gamma(0))}{t} + \frac{K}{2}(1 - t) d_g^2(\gamma(0), \gamma(1)).
			\]
			Taking \( t \to 0^+ \) and recalling (1) yields \eqref{variation of sc}.
		\end{enumerate}
	\end{Rem}
	
	\begin{Lem}[\cite{Cannarsa_Yu2009}]\label{uniform semiconcave}
		Let \( U \subset \mathbb{R}^n \) be an open set. If \(\{\psi_i \mid \psi_i: U \rightarrow \mathbb{R}, i \in \mathbb{N}\}\) is a sequence of \(K\)-semiconcave functions converging uniformly to a function \(\psi: U \rightarrow \mathbb{R}\), then \(\psi\) is \(K\)-semiconcave. Furthermore, if \(\{x_i\} \subset U\) is a sequence converging to \(x \in U\), then for every limit point \(p\) of \(\{p_i\}\) with \(p_i \in D^{+} \psi_i(x_i)\), we have \(p \in D^{+} \psi(x)\).
	\end{Lem}

	\begin{Lem}[{\cite[Lemma 2.1]{Cannarsa_Yu2009}}]\label{modify}
		Let \( U \subset \mathbb{R}^n \) be an open set such that $U\subset 
			\bar U \subset \Omega$ with $\Omega$ a bounded open subset of $\R^n$ and let $u:\Omega\to \R^n$ be a semiconcave function. For each $p_0\in D^+u(x_0)$ with   $x_0\in U$, there exists a sequence of smooth modifiers $\{\eta^m \mid m\in \N^+\}$ such that $\eta^m \geq 0$ with $\operatorname{supp}(\eta^m) \subset B(0,1/m)$ and $\int \eta^m dx = 1$, as well as the modified sequence $$\biggr\{u^m(x)= \eta^m*u(x)=\int \eta^m(y)u(x-y)dy:  m\in \N^+ \biggr\}$$ 
			satisfying
			\begin{enumerate}[\rm(1)]
				\item $\lim\limits_{m\rightarrow \infty} u^m(\cdot ) =u(\cdot )$ uniformly in $\bar U$,
				\item $\max\limits_{x\in \bar U}  |D u^m(x)|\leq C, \max\limits_{x\in \bar U}  |D^2 u^m(x)|\leq C$,
				\item $\lim\limits_{m\rightarrow \infty} D u^m(x_0)=p_0$.
		\end{enumerate}
	\end{Lem}
	
\begin{Pro}[{\cite[Theorem 14.1]{Villani_book2009}}]\label{thm:second_differentiability}
		Let $\psi \colon M \to \mathbb{R}$ be a locally geodesically semiconcave function on a Riemannian manifold $(M,g)$. Then for almost every $x \in M$, we have:
		\begin{enumerate}[\rm (1)]
			\item $\psi$ is differentiable at $x$;
			\item There exists a symmetric operator $A \colon T_xM \to T_xM$ satisfying either of the following equivalent conditions:
			\begin{enumerate}[\rm (a)]
				\item For any $v \in T_xM$, $\nabla_v (\nabla \psi)(x) = A v$ \footnote{For any vector field $\xi$ with $\xi(y) \in \nabla^- \psi(y)$ for all $y$, we have $\nabla_v \xi(x) = A v$.};
				\item $\psi$ admits the second-order expansion:
				\[
				\psi(\exp_x v) = \psi(x) + \langle \nabla \psi(x), v \rangle + \frac{1}{2} \langle A v, v \rangle + o(|v|^2) \quad \text{as } v \to 0.
				\]
			\end{enumerate}
		\end{enumerate}
		The operator $A$ is denoted by $\nabla^2 \psi(x)$ and called the Hessian of $\psi$ at $x$. The associated quadratic form may also be referred to as the Hessian when no confusion arises.
\end{Pro}

	\section{Global propagation of singularities}\label{sec:gps}
	
	This section is devoted to the study of global singularity propagation, with particular focus on both Riemannian metrics and magnetic Lagrangians.
	
	Before proving the invariance of under the generalized gradient flow \eqref{eq:ggf_intro}, we first establish the uniqueness of solutions to the following differential inclusion associated with magnetic Lagrangians.
		
	\begin{Pro}\label{pro:uniqueness}
		Given a magnetic Lagrangian \( L(x,v) = \frac{1}{2} g_x(v,v) - \omega(v) - V(x) \) and a locally semiconcave function $u$, the local Lipschitz solution to the differential inclusion 
			\begin{equation}\label{differential inclusion}		
				\dot\gamma(t)\in \{ H_p(\gamma(t),p) \mid p\in D^+u(\gamma(t))\}, t\in[0,+\infty)\, a.e., \quad  \gamma(0)=x_0	\end{equation}
			is unique, where $H$ is the Hamiltonian associated with $L$. In particular,  the solution to the  generalized gradient flow defined in \eqref{eq:ggf_intro} is unique.
	\end{Pro}
	
	\begin{proof}
 The existence of a Lipschitz curve satisfying \eqref{differential inclusion} is proved in \cite{CCHW2024,Cannarsa_Yu2009, Yu2006}.
		Suppose there exist $x_0 \in M$ and $t > 0$ such that two distinct Lipschitz curves $\gamma_1, \gamma_2 \colon [0,t] \to M$ 
		satisfy 
		$$  \dot\gamma_i(t)\in \{ H_p(\gamma_i(t),p) \mid p\in D^+u(\gamma_i(t))\}, t\geq0, \quad  \gamma_i(0)=x_0, i=1,2    $$
		and $ \sup_{s\in[0,t]} d_g(\gamma_1(s), \gamma_2(s)) > 0$.
		
		By reducing $t>0$ if necessary, we may assume $\gamma_i([0,t]) \subset U_x$ for $i = 1,2$, where $U_x$ is a geodesically convex neighborhood of $x$, and that $\sup_{s\in[0,t]} d(\gamma_1(s), \gamma_2(s)) \leq 1$.
		
		For each \( s \in (0,t] \), let \(\eta \colon [0,1] \to U\) be a geodesic satisfying \(\eta(0) = \gamma_1(s)\) and \(\eta(1) = \gamma_2(s)\). Then
		\begin{align*}
			\frac{d}{ds} d_g(\gamma_1(s), \gamma_2(s)) 
			&= \bigl\langle \nabla_1 d_g(\gamma_1(s), \gamma_2(s)), \dot{\gamma}_1(s) \bigr\rangle 
			+ \bigl\langle \nabla_2 d_g(\gamma_1(s), \gamma_2(s)), \dot{\gamma}_2(s) \bigr\rangle \\
			&= \bigl\langle -\dot{\eta}(0), \dot{\gamma}_1(s) \bigr\rangle 
			+ \bigl\langle \dot{\eta}(1), \dot{\gamma}_2(s) \bigr\rangle \\
			&= \bigl\langle -\dot{\eta}(0), v_1 + X(\eta(0)) \bigr\rangle 
			+ \bigl\langle \dot{\eta}(1), v_2 + X(\eta(1)) \bigr\rangle \\
			&= \bigl\langle -\dot{\eta}(0), v_1 \bigr\rangle 
			+ \bigl\langle \dot{\eta}(1), v_2 \bigr\rangle \\
			&\quad + \bigl\langle -\dot{\eta}(0), X(\eta(0)) \bigr\rangle 
			+ \bigl\langle \dot{\eta}(1), X(\eta(1)) \bigr\rangle,
		\end{align*}
		where $v_1\in \nabla^+u(\eta(0))$ and $v_2\in \nabla^+u(\eta(1))$.
		
		Since \( u \) is locally semiconcave, by Remark~\ref{basic property of sc} (4) we have
		\[
		u(\eta(0)) - u(\eta(1)) \leq \bigl\langle v_2, -\dot{\eta}(1) \bigr\rangle + \frac{k}{2} d_g^2(\eta(0), \eta(1)),
		\]
		\[
		u(\eta(1)) - u(\eta(0)) \leq \bigl\langle v_1, \dot{\eta}(0) \bigr\rangle + \frac{k}{2} d_g^2(\eta(0), \eta(1)).
		\]
		Adding these inequalities yields
		\[
		\bigl\langle v_1, -\dot{\eta}(0) \bigr\rangle + \bigl\langle v_2, \dot{\eta}(1) \bigr\rangle \leq k d_g^2(\eta(0), \eta(1)) \leq k d_g(\eta(0), \eta(1)),
		\]
		where the last inequality holds since \( d_g(\eta(0), \eta(1)) \leq 1 \) by our earlier assumption. Furthermore,
		\[
		\begin{aligned}
			\bigl\langle -\dot{\eta}(0), X(\eta(0)) \bigr\rangle + \bigl\langle \dot{\eta}(1), X(\eta(1)) \bigr\rangle 
			& = \bigl\langle -\dot{\eta}(1),P_{\eta(0),\eta(1)} (X(\eta(0))) \bigr\rangle+ \bigl\langle \dot{\eta}(1), X(\eta(1)) \bigr\rangle   \\
			&= \bigl\langle \dot{\eta}(1), X(\eta(1)) - P_{\eta(0),\eta(1)} (X(\eta(0))) \bigr\rangle \\
			&\leq \tilde{k} d_g(\eta(0), \eta(1)),
		\end{aligned}
		\]
		where \( P_{\eta(0),\eta(1)} (X(\eta(0))) \in T_{\eta(1)} M \) denotes the parallel transport of \( X(\eta(0)) \) along \( \eta \).
		
		Therefore, we have
		\[
		\frac{d}{ds} d_g(\gamma_1(s), \gamma_2(s)) \leq C \, d_g(\gamma_1(s), \gamma_2(s)).
		\]
		By Gronwall's inequality with the initial condition \( d_g(\gamma_1(0), \gamma_2(0)) = 0 \), we conclude that
		\begin{align*}
			d_g(\gamma_1(s), \gamma_2(s)) = 0,\qquad\forall s \in [0,t].
		\end{align*}
		This contradicts the assumption that \( \sup_{s\in[0,t]} d_g(\gamma_1(s), \gamma_2(s))>0 \).
	\end{proof}
	
	\subsection{Riemannian case}\label{sec:Riemannian}
	
	Let $(M,g)$ be a connected complete Riemannian manifold and $u \colon W \subset M \to \mathbb{R}$ be a viscosity solution to the Eikonal equation
	\[
	g(\nabla u(x),\nabla u(x)) =\langle \nabla u,\nabla u \rangle(x)= 1, \quad x \in W,
	\]
	where $W$ is an open subset of $M$ and $\langle \cdot,\cdot\rangle$ is the Riemannian inner product. In this subsection, we assume throughout that \(u\) is a viscosity solution of the Eikonal equation mentioned above. A Lipschitz curve \(\gamma(t):[0,T]\to M\) is said to be a generalized gradient flow if it satisfies
\[
\dot\gamma(t) = \nabla^\# u(\gamma(t)) = (p^\#_u)^\sharp(\gamma(t)), \quad \text{for a.e. } t \in [0,T],
\]
where $p_u^\#(x) := \underset{p \in D^+u(x)}{\arg\min}\, \frac{1}{2}g_x^*(p,p)= \underset{p \in D^+u(x)}{\arg\min}\, g_x^*(p,p)$.

Note that, by the standard theory of viscosity solutions, \(u\) is locally semiconcave on \(W\).

	\begin{Lem}\label{lem:hessian_property}
		For any $x \in W$ where $u$ is twice differentiable, the Hessian of $u$ satisfies:
		\[
		\operatorname{Hess} u(x)(\xi, \nabla u(x)) = 0 \quad \text{for all } \xi \in T_x M.
		\]
		Moreover, for any bounded open subset $U \subset W$, there exists a constant $K = K(U) > 0$ such that at all points $x \in U$ where $u$ is twice differentiable,
		\[
		\operatorname{Hess} u(x)(\xi, \xi) \leq K\left(\langle \xi, \xi \rangle - \langle \xi, \nabla u(x) \rangle^2\right) \quad \text{for all } \xi \in T_x M.
		\]
	\end{Lem}
	
	\begin{proof}
		For any $\xi \in T_x M$, Proposition \ref{thm:second_differentiability} yields
		\[
		\nabla_{\xi} \langle \nabla u, \nabla u \rangle = 2 \langle \nabla_{\xi} \nabla u, \nabla u \rangle = 2 \operatorname{Hess} u(\xi, \nabla u).
		\]
		Since $u$ satisfies the Eikonal equation $\langle \nabla u, \nabla u \rangle \equiv 1$, its derivative vanishes, giving
		\[
		\operatorname{Hess} u(\xi, \nabla u) = 0.
		\]
		
		Since $u$ is locally semiconcave, there exists a constant $K > 0$ such that for all $x \in U$ where $u$ is twice differentiable and all $\xi \in T_x M$,
		\[
		\operatorname{Hess} u(\xi, \xi) \leq K \langle \xi, \xi \rangle.
		\]
		Decompose $\xi$ into its components parallel and perpendicular to $\nabla u$
		\[
		\xi^{\perp} := \xi - \langle \xi, \nabla u \rangle \nabla u.
		\]
		Using the orthogonality property $\operatorname{Hess} u(\xi, \nabla u) = 0$, we compute
		\[
		\begin{aligned}
			\operatorname{Hess} u(\xi, \xi) &= \operatorname{Hess} u(\xi^{\perp} + \langle \xi, \nabla u \rangle \nabla u, \xi^{\perp} + \langle \xi, \nabla u \rangle \nabla u) \\
			&= \operatorname{Hess} u(\xi^{\perp}, \xi^{\perp}) + 2\langle \xi, \nabla u \rangle \operatorname{Hess} u(\xi^{\perp}, \nabla u) + \langle \xi, \nabla u \rangle^2 \operatorname{Hess} u(\nabla u, \nabla u) \\
			&= \operatorname{Hess} u(\xi^{\perp}, \xi^{\perp}) \\
			&\leq K \langle \xi^{\perp}, \xi^{\perp} \rangle \\
			&= K \left( \langle \xi, \xi \rangle - \langle \xi, \nabla u \rangle^2 \right).
		\end{aligned}
		\]
	\end{proof}
	
	\begin{The}[Global Propogation of Singularities]\label{thm:singularity_persistence}
		Let $U \subset W$ be a bounded open subset and $x_0 \in U$. If $x_0 \in \SING(u)$, then the generalized gradient flow $\gamma:[0,T] \to U$ with $\gamma(0) = x_0$ satisfies
		\[
		\gamma(t) \in \SING(u) \quad \text{for all } t \in [0,T],
		\]
		where $\SING(u)$ denotes the singularity sets of $u$ in $W$.
	\end{The}
	
	\begin{proof}
		We work in local coordinates $(x^1,\ldots,x^n)$ on $U$ to establish the key estimates with
		\begin{itemize}
			\item Metric matrix $( g_{ij})_{n\times n} = ( g(\partial_i,\partial_j))_{n\times n}$ and it's inverse matrix $( g^{ij})_{n \times n} = ( g(\partial_i,\partial_j))_{n\times n}^{-1}$
			\item Christoffel symbols $\Gamma_{ij}^k$
		\end{itemize} From the Hessian bound
		\[
		\operatorname{Hess} u(\xi, \xi) \leq K(\langle \xi, \xi \rangle - \langle \xi, \nabla u \rangle^2),
		\]
		valid at all twice differentiable points $x \in U$, we obtain in coordinates
		\begin{equation}\label{original hessian}
			\xi^i \left( u_{ij}(x) - \Gamma_{ij}^k(x) u_k(x) \right) \xi^j \leq K \left( \xi^i g_{ij}(x) \xi^j - (\xi^i g_{ij}(x) u^j(x))^2 \right),
		\end{equation}
		where $\xi = \xi^i \partial_i$, $u_{ij} = \frac{\partial^2 u}{\partial x^i \partial x^j}$,  $u_k=\frac{\partial u}{\partial x^k}$ and $u^i=g^{ij}u_j$. (Here and hereafter, we use the Einstein summation convention.)
		
		In view of Lemma \ref{modify}, we may choose a sequence of mollifiers \(\{\eta^m\}\) consisting of nonnegative smooth bump functions such that \(\text{supp}\, \eta^m \subset B(0,1/m)\) and \(\int \eta^m  dx = 1\), and then define the corresponding regularized functions.
		\[
		u^m(x) := (\eta^m * u)(x),
		\]
		which satisfy:
		\begin{itemize}
			\item $|Du^m| \leq \text{Lip}\,(u)$,
			\item $u^m$ are uniformly semiconcave,
			\item $u^m \to u$ uniformly on compact sets,
			\item $Du^m(x_0) \to p_0\in D^+u(x_0)$ with $\|p_0 \|_{g^*}<1$.
		\end{itemize}
		The mollified Hessian satisfies
		\[
		\begin{aligned}
			\xi^i \left( u_{ij}^m - \Gamma_{ij}^k u_k^m \right) \xi^j &= \left[ \eta_m * (u_{ij} - \Gamma_{ij}^k u_k) \right] \xi^i \xi^j \\
			&\quad + \int \eta_m(x-y) \left( \Gamma_{ij}^k(y) - \Gamma_{ij}^k(x) \right) u_k(y) \xi^i \xi^j dy \\
			&=: I_m + II_m.
		\end{aligned}
		\]
		
		The error term $II_m$ is estimated using the Cauchy-Schwarz inequality
		\[\begin{aligned}
			II_m &= \int (\eta^m)^{\frac{1}{2}} \left( \Gamma_{ij}^k(y) - \Gamma_{ij}^k(x) \right) (\eta^m)^{\frac{1}{2}} u_k \xi^i \xi^j dy\\
			&\leq \sum_{i,j,k} \left( \int \eta^m(x-y) \left( \Gamma_{ij}^k(y) - \Gamma_{ij}^k(x) \right)^2 dy \right)^{\frac{1}{2}}\times \left( \int \eta^m(x-y) \left( u_k \xi^i \xi^j \right)^2 dy \right)^{\frac{1}{2}} \\
			&= o\left( \frac{1}{m} \right).
		\end{aligned}
		\]
		
		For the main term $I_m$, we apply the inequality \eqref{original hessian} to get
		\[
		I_m \leq K \left( \xi^i g_{ij}(x) \xi^j +o\left( \frac{1}{m} \right) - \int \eta_m(x-y) (\xi^i g_{ij}(y) u^j(y))^2 dy \right).
		\]
		Using Jensen's inequality for the second term
		\[
		\begin{aligned}
			\int \eta^m(x-y) \left( \xi^i g_{ij}(y) u^j(y) \right)^2 dy &= \int (\eta^m)^{\frac{1}{2}\times 2}(x-y) dy \int (\eta^m)^{\frac{1}{2}\times 2}(x-y) \left( \xi^i g_{ij}(y) u^j(y) \right)^2 dy \\
			&\geq \left( \int \eta^m(x-y) \xi^i g_{ij}(y) u^j(y) dy \right)^2 \\
			&= \left( \int \eta^m(x-y) \xi^i u_i(y) dy \right)^2\\
			&= \left\{ \xi^i (\eta^m * u_i)(x) \right\}^2 = \langle \xi, \nabla u^m \rangle^2,
		\end{aligned}
		\]
		we obtain the regularized version
		\[
		\operatorname{Hess} u^m(\xi, \xi) \leq K \left( \langle \xi,\xi \rangle - \langle \xi, \nabla u^m \rangle^2 \right) + o\left( \frac{1}{m} \right).
		\]
		
		Now consider the gradient flow $\gamma_m$ of $u^m$
		\[
		\begin{cases}
			\dot\gamma_m(t) = \nabla u^m(\gamma_m(t)), \\
			\gamma_m(0) = x_0.
		\end{cases}
		\]
		Define $\psi_m(t) := \langle \nabla u^m(\gamma_m(t)), \nabla u^m(\gamma_m(t)) \rangle - 1=\|Du_m(\gamma_m(t)) \|_{g^*}^2-1$. Differentiating yields
		\[
		\begin{aligned}
			\psi_m'(t) &= \operatorname{Hess} u^m(\nabla u^m(\gamma_m(t)), \nabla u^m(\gamma_m(t))) \\
			&\leq K \left\{ \langle \nabla u^m, \nabla u^m \rangle - \langle \nabla u^m, \nabla u^m \rangle^2 \right\} + o\left( \frac{1}{m} \right) \\
			&= K \left\{ \langle \nabla u^m, \nabla u^m \rangle (1 - \langle \nabla u^m, \nabla u^m \rangle) \right\} + o\left( \frac{1}{m} \right) \\
			&\leq -C \psi_m(t) + o\left( \frac{1}{m} \right)
		\end{aligned}
		\]
		for some constant $C > 0$. So Gronwall's inequality gives
		\[
		\psi_m(T) \leq e^{-CT}\psi_m(0) + \int_0^T e^{C(s-T)} o\left( \frac{1}{m} \right) ds.
		\]
		Since $\psi_m(0)  =\|Du_m(x_0) \|_{g^*}^2-1 \to \|p_0 \|_{g^*}-1=a<0 $, taking $m \to \infty$, the $o\left( \frac{1}{m} \right)$ terms vanish and we conclude $\liminf\limits_{m\to \infty} \psi_m(t) < 0$ for all $t$.
		
		By the Arzel\`a--Ascoli theorem, we may assume, up to a subsequence, that the curves \(\gamma_m(t)\) converge uniformly to a Lipschitz curve. From Lemma \ref{uniform semiconcave}, we conclude that this limiting curve is a solution to the differential inclusion \eqref{differential inclusion}. Combining this result with Proposition \ref{differential inclusion}, we deduce that the entire sequence \(\gamma_m\) converges uniformly to \(\gamma\). Define \(\bar H(x,p) = \|p\|_{g^*}^2 - 1\). Then, \(\gamma(t) \in \SING(u)\) if and only if \(\bar H\bigl(\gamma(t), p^\#_u(\gamma(t))\bigr) <0\). By Lemma \ref{uniform semiconcave}, we have that 
		$$  0>\liminf\limits_{m\to \infty} \psi_m(t)=\liminf\limits_{m\to \infty} \bar H(\gamma_m(t),Du^m(\gamma_m(t)))\geq \bar H(\gamma(t),p^\#_u(\gamma(t))), $$
		proving $\gamma(t) \in \SING(u)$.
	\end{proof}
	
	\begin{Rem}\label{rem:Finsler}
		The preceding proof fails when \( M \) is equipped with a Finsler metric \( F \) and \( u \colon U \subset M \to \mathbb{R} \) is a viscosity solution of the Finslerian eikonal equation
		\[
		F(\nabla u, \nabla u) = 1, \qquad x \in W,
		\]
		where \( \nabla u \) denotes the Finslerian gradient. The obstruction arises because the Finsler metric matrix \( \bigl( g_{ij}(x, \nabla u) \bigr)_{1 \leq i,j \leq n} \) depends nonlinearly and hence sensitively on the gradient \(\nabla u\). Consequently, for a mollified sequence \( u^{m} = \eta^{m} * u \), the matrices \( \bigl( g_{ij}(x, \nabla u^m) \bigr)_{1 \leq i,j \leq n} \) need not converge uniformly on compact subsets to \( \bigl( g_{ij}(x, \nabla u) \bigr)_{1 \leq i,j \leq n} \), which invalidates the previous argument.
	\end{Rem}
	
	\subsection{Magnetic case}\label{sec:magnetic}
	
	Let $(M,g)$ be a connected closed Riemannian manifold and $\omega$ a smooth 1-form (not necessarily closed). Consider the Lagrangian $L\colon TM \to \mathbb{R}$ defined by
	\[
	L(x,v) = \frac{1}{2}g_x(v,v) - \omega_x(v) - V(x),
	\]
	where $V\in C^2(M)$ is a  potential function. The corresponding Hamiltonian $H \colon T^*M \to \mathbb{R}$ is
	\[
	H(x,p) = \frac{1}{2}g_x^*(p + \omega_x, p + \omega_x) + V(x),
	\]
	where $g^*$ denotes the dual Riemannian metric on $T^*M$. The Hamilton-Jacobi equation for this system is:
	\begin{equation}\label{eq:HJ}
		H(x, Du) = \frac{1}{2}\|Du + \omega\|^2_{g^*} + V(x)=\frac{1}{2}\langle \nabla u+X,\nabla u+X \rangle+V(x) = c,
	\end{equation}
	where $u \colon M \to \mathbb{R}$ is a weak KAM solution, $c = c[L]$ is the critical value, $\|\cdot\|_{g^*}$ is the dual norm induced by $g$, $\langle \cdot, \cdot \rangle$ denotes the Riemannian inner product, $\nabla u$ represents the Riemannian gradient and  $X = \omega^\sharp$ is the vector field dual to the 1-form $\omega$ via the musical isomorphism.
	
	\begin{Lem}\label{lem:critical_value}
		For the Lagrangian $L(x,v) = \frac{1}{2}g_x(v,v) - \omega_x(v) - V(x)$ with $\max V = 0$, the critical value satisfies $c[L] \geq 0$.
	\end{Lem}
	
	\begin{proof}
		Let $x_0 \in M$ be a point where $V(x_0) = 0$ (which exists by our normalization). Consider the constant path $\gamma \colon [0,1] \to M$ defined by $\gamma(t) \equiv x_0$. For this path, we compute
		\[
		\int_0^1 L(\gamma(t), \dot\gamma(t))\,dt = \int_0^1 \left(\frac{1}{2}\|\dot\gamma(t)\|_{g_{x_0}}^2 - \omega_{x_0}(\dot\gamma(t)) - V(\gamma(t))\right)dt = 0,
		\]
		where the vanishing occurs since $\dot\gamma(t) \equiv 0$ and $V(x_0) = 0$.
		
		According to the weak KAM theorem, there exists a viscosity solution $u \in C(M,\mathbb{R})$ of the Hamilton-Jacobi equation that is dominated by $L + c[L]$, meaning that for every absolutely continuous curve $\eta \colon [0,T] \to M$ with endpoints $\eta(0) = x$ and $\eta(T) = y$, the following inequality holds
		\[
		u(y) - u(x) \leq \int_0^T \left[L(\eta(t), \dot\eta(t)) + c[L]\right] dt.
		\]
		
		Applying this inequality to the constant path $\gamma \equiv x_0$ (with $x = y = x_0$ and $T = 1$), we obtain that
		\[
		0 = u(x_0) - u(x_0) \leq \int_0^1 L(\gamma(t), \dot\gamma(t))\,dt + c[L] = c[L],
		\]
		where we used the fact that $\int_0^1 L(\gamma(t), \dot\gamma(t))\,dt = 0$ as computed previously. This establishes the lower bound $c[L] \geq 0$.
	\end{proof}
	
	\begin{Lem}\label{lem:non_exact_critical}
		Consider the Lagrangian $L(x,v) = \frac{1}{2}g_x(v,v) - \omega_x(v)$, where $\omega$ is a non-exact $1$-form. Then the critical value satisfies the strict inequality $c[L] > 0$.
	\end{Lem}
	
	\begin{proof}
		Since \(\omega\) is not exact, de Rham's theorem ensures the existence of a closed curve
		\(\gamma\colon [0,1]\to M\) such that
		\[
		C_1:=\int_{\gamma}\omega>0.
		\]
		Define the (positive) energy of this loop by
		\[
		C_2:=\frac12\int_0^1 g_{\gamma(t)}\!\bigl(\dot\gamma(t),\dot\gamma(t)\bigr)\,dt>0.
		\]
		
		
		Reparameterizing the original curve by
		\[
		\gamma_s(t):=\gamma(st),\qquad t\in[0,1/s],
		\]
		its action becomes
		\[
		\mathcal A(\gamma_s)=\int_{\gamma_s}L=sC_2-C_1.
		\]
		Whenever the scaling parameter satisfies \(s\in(0,C_1/C_2)\), we immediately have \(\mathcal A(\gamma_s)<0\).
		
		By the weak KAM theorem, there exists a weak KAM solution \(u\in C(M,\mathbb R)\) satisfying the domination inequality
		\[
		u(y)-u(x)\le \int_0^T L\!\bigl(\eta(t),\dot\eta(t)\bigr)\,dt+c[L]\,T
		\]
		for every absolutely continuous path \(\eta\colon[0,T]\to M\) with \(\eta(0)=x\) and \(\eta(T)=y\).
		
		%
		
		Choosing \(s_1=C_1/(2C_2)\) gives
		\[
		\mathcal A(\gamma_{s_1})=s_1C_2-C_1=-\frac{C_1}{2}.
		\]
		Since \(\gamma_{s_1}\) is closed, \(\gamma_{s_1}(0)=\gamma_{s_1}(1/s_1)\).  The weak KAM inequality then yields
		\[
		0=u\bigl(\gamma_{s_1}(1/s_1)\bigr)-u\bigl(\gamma_{s_1}(0)\bigr)
		\le \mathcal A(\gamma_{s_1})+\frac{c[L]}{s_1}.
		\]
		Rearranging, we obtain
		\[
		c[L]\ge -s_1\mathcal A(\gamma_{s_1})
		=\frac{C_1}{2C_2}\cdot\frac{C_1}{2}>0. \qedhere
		\]
		%
	\end{proof}
	
	The following proposition demonstrates the prevalence of singularities in $u$.
	
	\begin{Lem}\label{trivial bundle}
		Let $(M,g)$ be a closed orientable surface. If there exists a continuous nonvanishing vector field $X$ on $M$, then the tangent bundle $TM$ and cotangent bundle $T^*M$ are trivial.
	\end{Lem}
	
	\begin{proof}
		Since a vector bundle $E$ is trivial precisely when it admits a global frame, it suffices to construct global frames for $TM$ and $T^*M$. Given the continuous nonvanishing vector field $X$, define its dual 1-form $X^\flat$ by $X^\flat(v) = g(X, v)$ for all $v \in TM$. Using the Hodge star operator $\star$, we obtain an orthogonal 1-form $\star X^\flat$ satisfying $g^*(X^\flat, \star X^\flat) = 0$, where $g^*$ is the induced metric on $T^*M$. (Geometrically, $\star$ acts on tangent vectors by rotation by $\pi/2$ in each oriented tangent plane.) Thus $\{X^\flat, \star X^\flat\}$ forms a global frame for $T^*M$, and consequently $\{X, (\star X^\flat)^\sharp\}$ is a global frame for $TM$.
	\end{proof}
	
	\begin{The}\label{thm:sing_exists}
		Let $(M,g)$ be a connected closed orientable surface with Euler characteristic $\chi(M) \neq 0$. Consider the Lagrangian 
		\[
		L(x,v) = \tfrac{1}{2}g_x(v,v) - \omega_x(v) - V(x)
		\]
		where $\max_{x\in M} V(x) = 0$, and let $u$ be a weak KAM solution of \eqref{eq:HJ}. If $\SING(u) = \emptyset$, then $c[L] = \max V = 0$. In particular, $\SING(u) \neq \emptyset$ when $V \equiv 0$ and $\omega$ is non-exact.
	\end{The}
	
	\begin{proof}
		Suppose $c[L] > \max V = 0$ and $\SING(u) = \emptyset$. Then $u \in C^1(M)$, and equation \eqref{eq:HJ} implies
		\[
		\langle \nabla u + X, \nabla u + X \rangle = 2(c[L] - V) > 0.
		\]
		Thus $\nabla u + X$ constitutes a continuous nonvanishing vector field on $M$. By Lemma \ref{trivial bundle}, this implies $TM$ is trivial. However, the condition $\chi(M) \neq 0$ contradicts the existence of a trivial tangent bundle. Therefore $\SING(u) \neq \emptyset$. The final statement follows from Lemma \ref{trivial bundle}.
	\end{proof}
	
	\begin{Lem}\label{lem:singular_set_energy_gap}
		For every singular point $x \in \SING(u)$, the following strict energy gap inequality holds
		\[
		c[L] - V(x) > 0.
		\]
	\end{Lem}
	
	\begin{proof}
		Let $x \in \SING(u)$ be a singular point. By definition of the singular set, there exist at least two distinct reachable differentials $p_1, p_2 \in D^*u(x)$ satisfying the Hamilton-Jacobi equation
		\[
		H(x, p_i) = \frac{1}{2}\|p_i + \omega_x\|^2_{g_x^*} + V(x) = c[L], \quad \text{for } i = 1,2,
		\]
		where $D^*u(x)$ denotes the reachable differential of $u$ at $x$, $\|\cdot\|_{g_x^*}$ is the dual norm induced by the Riemannian metric $g$ at $x$, $\omega_x$ is the evaluation of the 1-form $\omega$ at $x$ and $c[L]$ is the critical value of the Lagrangian $L$.
		
		Consider the convex combination of the shifted covectors
		\[
		\tilde{p} = t(p_1 + \omega) + (1 - t)(p_2 + \omega) = t p_1 + (1 - t) p_2 + \omega,
		\]
		where \( t \in (0,1) \). Using the strict convexity of the Hamiltonian's kinetic term \( \frac{1}{2} g^*(\cdot, \cdot) \), we obtain the strict inequality
		\[
		\begin{aligned}
			\frac{1}{2} g^*(\tilde{p}, \tilde{p}) + V(x) 
			&< t \left( \frac{1}{2} g^*(p_1 + \omega, p_1 + \omega) \right) + (1 - t) \left( \frac{1}{2} g^*(p_2 + \omega, p_2 + \omega) \right) + V(x) \\
			&= t \left( c[L] - V(x) \right) + (1 - t) \left( c[L] - V(x) \right) + V(x) \\
			&= c[L],
		\end{aligned}
		\]
		where we have used the Hamilton-Jacobi equation \( \frac{1}{2} g^*(p_i + \omega, p_i + \omega) = c[L] - V(x) \) for \( i = 1,2 \). This inequality implies the following energy gap estimate
		\[
		c[L] - V(x) > \frac{1}{2}g^*(\tilde{p}, \tilde{p}) \geq 0.  \qedhere
		\]
	\end{proof}
	
	For the magnetic Lagrangian $L(x,v) = \frac{1}{2}g_x(v,v) - \omega_x(v) - V(x)$ on $(M,g)$, define
	\[
	W := \{x \in M \mid c[L] > V(x)\}.
	\]
	On $W$, the Hamilton-Jacobi equation has equivalent forms:
	\begin{align}
		&\frac{1}{2}\|Du + \omega\|^2_{g^*} + V = c[L], \label{eq:HJ1}\tag{HJ$_1$} \\
		&\frac{\|Du + \omega\|^2_{g^*}}{2(c[L]-V)} = 1. \label{eq:HJ2}\tag{HJ$_2$}
	\end{align}
	
	The rescaled metric $\tilde{g}_x = (c[L]-V(x))g_x$ induces the dual metric $\tilde{g}^*_x = g^*_x/(c[L]-V(x))$. In this geometry, the Hamilton-Jacobi equation becomes
	\begin{equation}\label{meq}
		\frac{1}{2}\|\nabla^{\tilde{g}} u + X^{\tilde{g}}\|^2_{\tilde{g}} = 1,
	\end{equation}
	where $X^{\tilde{g}}$ is the $\tilde{g}$-dual of $\omega$ and $\nabla^{\tilde{g}}u$ is the $\tilde{g}$-gradient.
	
	Let $f(x) := c[L] - V(x)$ for $x \in W$ be the energy gap function. Consider the modified Lagrangian 
	\[
	\tilde{L}(x, v) := \frac{1}{2}\tilde{g}_x(v, v) - \omega_x(v) = \frac{f(x)}{2}g_x(v, v) - \omega_x(v),
	\]
	with the Hamiltonian defined by
	\[
	\tilde{H}(x, p) = \sup_{v \in T_x M} \bigl\{ \langle p, v \rangle - \tilde{L}(x, v) \bigr\} = \frac{1}{2}\tilde{g}^*_x(p + \omega_x, p + \omega_x).
	\]
	For a weak KAM solution $u$, we obtain
	\begin{align}
		\tilde{H}(x, Du(x)) &= \frac{1}{2}\tilde{g}^*_x(Du(x) + \omega_x, Du(x) + \omega_x) \notag \\
		&= \frac{1}{2} \tilde{g}_x(\nabla^{\tilde{g}} u(x) + X(x), \nabla^{\tilde{g}} u(x) + X(x)) = 1,
	\end{align}
	where $\nabla^{\tilde{g}}u$ denotes the gradient with respect to $\tilde{g}$. The Hamiltonian vector fields are related by
	\[
	\tilde{H}_p(x,p) = \frac{1}{f(x)}H_p(x,p),
	\]
	with the original Hamiltonian given by
	\[
	H(x, p) = \frac{1}{2}g^*_x(p + \omega_x, p + \omega_x) + V(x).
	\]
	
	We now introduce two distinct gradient flow systems on $W$. Define the minimal selection
	\[
	p^\#_u(x) := \arg\min\{H(x,p) \mid p \in D^+u(x)\}.
	\]
	The first gradient flow system is
	\begin{equation}\label{G1}
		\begin{cases}
			\dot{x}(t) = H_{p}(x(t), p_{u}^{\#}(x(t))) \\
			x(0) = x_0 \in W.
		\end{cases} \tag{G1}
	\end{equation}
	The second gradient flow system incorporates the conformal scaling factor $1/f(x)$ and takes the form
	\begin{equation}\label{G2}
		\begin{cases}
			\dot{x}(t) = \tilde{H}_{p}\bigl(x(t), p_{u}^{\#}(x(t))\bigr) = \dfrac{1}{f(x(t))} H_{p}\bigl(x(t), p_{u}^{\#}(x(t))\bigr), \\
			x(0) = x_0 \in W.
		\end{cases} \tag{G2}
	\end{equation}
	This system represents the gradient flow with respect to the conformally modified metric $\tilde{g}$.
	
	\begin{Pro}\label{prop:reparam}
		The solution curves of \eqref{G1} and \eqref{G2} coincide as sets in $W$, with only their time parameterizations differing through the energy gap factor $f(x) = c[L]-V(x)$. 
	\end{Pro}
	
	\begin{proof}
		Given a solution $x(t)$ of \eqref{G1}, we construct a solution to \eqref{G2} via time reparameterization. Define the rescaled time variable
		\[
		s(t) = \int_{0}^{t} f(x(\tau))\,d\tau = \int_{0}^{t} \big(c[L] - V(x(\tau))\big)\,d\tau,
		\]
		which is strictly increasing since $f(x)>0$ on $W$. Let $t(s)$ denote the inverse function and set $y(s) := x(t(s))$.
		
		Differentiating the reparameterized curve yields
		\begin{align*}
			\frac{dy}{ds} &= \frac{dx}{dt}\Big|_{t=t(s)} \cdot \frac{dt}{ds} \\
			&= \frac{1}{f(x(t(s)))} H_{p}\big(x(t(s)), p_{u}^{\#}(x(t(s)))\big) \quad \text{(using $\frac{dt}{ds} = \frac{1}{f(x(t(s)))}$)} \\
			&= \tilde{H}_{p}\big(y(s), p_{u}^{\#}(y(s))\big) \quad \text{(by definition of $\tilde{H}$)}.
		\end{align*}
		Thus $y(s)$ satisfies \eqref{G2} with $y(0) = x_0$, proving the trajectories coincide up to parameterization.
		
		For the converse, given $y(s)$ solving \eqref{G2}, define $t(s) = \int_0^s \frac{d\sigma}{f(y(\sigma))}$. Then
		\[
		\frac{d}{dt}y(s(t)) =\frac{dy}{ds}\Big|_{s=s(t)} \cdot\frac{ds}{dt} = H_p(y(s(t)),p_u^\#(y(s(t)))),
		\]
		showing $y(s(t))$ solves \eqref{G1}.
	\end{proof}
	
	We now analyze the modified Hamilton-Jacobi equation associated with the rescaled dynamics
	\begin{equation}\label{eq:HJ-modified}
		\tilde{H}(x, Du) = \frac{1}{2}\tilde{g}^*_x(Du + \omega_x, Du + \omega_x) 
		= \frac{1}{2} \tilde{g}_x(\nabla^{\tilde{g}} u + X, \nabla^{\tilde{g}} u + X) = 1, \quad x \in W,
	\end{equation}
	where
	\begin{itemize}  
		\item $u \in C(M,\mathbb{R})$ is a weak KAM solution of the original Hamilton-Jacobi equation,
		\item $\tilde{g}^*_x$ denotes the dual metric of $\tilde{g}_x = (c[L]-V(x))g_x$ at $x \in W$,
		\item $\nabla^{\tilde{g}} u$ is the gradient of $u$ with respect to $\tilde{g}$,
		\item $X =X^{\tilde g}= \omega^\sharp_{\tilde{g}}$ is the vector field dual to $\omega$ via the $\tilde{g}$-musical isomorphism.
	\end{itemize}
	
	\begin{Lem}\label{lem:hessian_estimate}
		Let $u$ be a weak KAM solution of \eqref{eq:HJ} and consider a point $x \in W$ where $u$ is twice differentiable. Then the following identity holds in the geometry of the conformal metric $\tilde{g}$
		\begin{equation}\label{eq:hessian-identity}
			(\operatorname{Hess}^{\tilde{g}} u + \nabla^{\tilde{g}} X)(\xi, \nabla^{\tilde{g}} u + X) = 0 \quad \forall \xi \in T_x M.
		\end{equation}
		Moreover, for any bounded open subset $U \subset W$, there exists a constant $\kappa = \kappa(U) > 0$ such that the following uniform estimate holds at all points \(x\in U\) of twice differentiability:
		\begin{equation}\label{eq:hessian-estimate}
			(\operatorname{Hess}^{\tilde{g}} u + \nabla^{\tilde{g}} X)(\xi, \xi) \leq \kappa \left( \|\xi\|_{\tilde{g}}^2 - \frac{1}{2} \langle \xi, \nabla^{\tilde{g}} u + X \rangle_{\tilde{g}}^2 \right)
		\end{equation}
		for all $\xi \in T_x M$.
	\end{Lem}
	
	\begin{proof}
		For any tangent vector $\xi \in T_x M$, since $u$ satisfies the modified Eikonal equation \eqref{meq}, Theorem~\ref{thm:second_differentiability} yields the orthogonality relation
		\[
		\langle \nabla^{\tilde{g}}_{\xi} (\nabla^{\tilde{g}} u + X), \nabla^{\tilde{g}} u + X \rangle_{\tilde{g}} = 0,
		\]
		where $\nabla^{\tilde{g}}$ denotes the Levi-Civita connection of the conformal metric $\tilde{g}$. This immediately implies the vanishing of the quadratic form
		\[
		(\operatorname{Hess}^{\tilde{g}} u + \nabla^{\tilde{g}} X)(\xi, \nabla^{\tilde{g}} u + X) = 0.
		\]
		
		Since $u$ is locally semiconcave, there exists a constant $\kappa = \kappa(U) > 0$ such that for all $x \in U$ where $u$ is twice differentiable,
		\begin{equation}\label{eq:basic-hessian-bound}
			(\operatorname{Hess}^{\tilde{g}} u + \nabla^{\tilde{g}} X)(\xi, \xi) \leq \kappa \langle \xi, \xi \rangle_{\tilde{g}}, \quad \forall \xi \in T_x M.
		\end{equation}
		Now we set
		\[
		\tilde{v} := \nabla^{\tilde{g}} u + X \quad \text{and} \quad \xi^\perp := \xi - \frac{\langle \xi, \tilde{v} \rangle_{\tilde{g}}}{|\tilde{v}|_{\tilde{g}}^2} \tilde{v} = \xi - \frac{1}{2}\langle \xi, \tilde{v} \rangle_{\tilde{g}} \tilde{v},
		\]
		where we used $|\tilde{v}|_{\tilde{g}}^2 = 2$ from equation \eqref{meq}. 
		
		Using the orthogonality relation $(\operatorname{Hess}^{\tilde{g}} u + \nabla^{\tilde{g}} X)(\cdot, \tilde{v}) = 0$, we obtain the refined estimate
		\begin{align*}
			(\operatorname{Hess}^{\tilde{g}} u + \nabla^{\tilde{g}} X)(\xi, \xi) &= (\operatorname{Hess}^{\tilde{g}} u + \nabla^{\tilde{g}} X)(\xi^\perp, \xi^\perp) \\
			&\leq \kappa \langle \xi^\perp, \xi^\perp \rangle_{\tilde{g}} \\
			&= \kappa \left( \langle \xi, \xi \rangle_{\tilde{g}} - \frac{1}{2} \langle \xi, \tilde{v} \rangle_{\tilde{g}}^2 \right). \qedhere
		\end{align*}
	\end{proof} 
	\begin{The}[Global Propogation of Singularity]\label{thm:singular_persistence}
		Let $(M,g)$ be a connected closed Riemannian manifold and $L(x,v) = \frac{1}{2}g_x(v,v) - \omega_x(v) - V(x)$ be a magnetic mechanical Lagrangian. For any weak KAM solution $u$ of \eqref{eq:HJ} and the bounded open subset $W = \{x \in M \mid c[L] > V(x)\}$, the following propagation property holds: for any initial point $x_0 \in W \cap \SING(u)$, the generalized gradient flow $\gamma(t)$ with $\gamma(0) = x_0$ satisfies:
		\[
		\gamma(t) \in \SING(u) \quad \text{for all } t \geq 0 \text{ where the flow is defined.}
		\]
		That is, the singular set $\SING(u)$ is invariant under the generalized gradient flow dynamics.
	\end{The}
	
	\begin{proof}
		We work in local coordinates $(x^1,\ldots,x^n)$ on $W$ with metric matrix $(\tilde g_{ij})_{n\times n} = (\tilde g(\partial_i,\partial_j))_{n\times n}$ and it's inverse matrix $(\tilde g^{ij})_{n \times n} = (\tilde g(\partial_i,\partial_j))_{n\times n}^{-1}$, Christoffel symbols $\tilde \Gamma_{ij}^k$, $\omega = \omega_j dx^j$, $X = X^j\partial_j$ where $X^j=\tilde  g^{jk} \omega_k$ and $\tilde{v} = \nabla^{\tilde g } u + X = (\tilde g^{ij}u_i + X^j)\partial_j = \tilde{v}^j\partial_j$.
		
		Since
		\[
		(\operatorname{Hess}^{\tilde g} u + \nabla^{\tilde g } X)(\xi, \xi) \leq \kappa \left( \langle \xi, \xi \rangle_{\tilde  g} - \frac{1}{2} \langle \xi, \tilde{v} \rangle^2_{\tilde  g} \right)
		\]
		for all \( x \in W \) where \( u \) is twice differentiable and each \( \xi \in T_x M \) by Lemma \ref{lem:hessian_estimate},
		we derive in coordinates:
		\begin{equation}\label{eq:Hessian-local}
			\xi^i\left(u_{ij} - \tilde \Gamma_{ij}^k u_k - \partial_i X^k \tilde g_{kj} + X^k \tilde \Gamma_{ik}^l \tilde g_{lj}\right)\xi^j \leq \kappa\left(\tilde g_{ij}\xi^i\xi^j - \frac{1}{2}(\tilde g_{ij}\xi^i\tilde{v}^j)^2\right).
		\end{equation}
		
		By Lemma \ref{modify}, we can take smooth modifiers $\eta^m \geq 0$ with $\text{supp}(\eta^m) \subset B(0,1/m)$ and $\int \eta^m dx = 1$  such that
		\[
		u^m(x) := (\eta^m \ast u)(x) = \int \eta^m(x-y)u(y)dy
		\]
		satisfying
		\begin{itemize}
			\item $|Du^m| \leq \text{Lip}(u)$, 
			\item $u^m$ are uniformly semiconcave, 
			\item $u^m \to u$ uniformly on compact sets,
			\item $Du^m(x_0) \to p_0\in D^+u(x_0)$ with $\frac{1}{2}\|p_0+\omega \|_{\tilde g^*}<1$.
		\end{itemize}
		
		For the regularized functions, we compute exactly:
		\begin{align*}
			&(\operatorname{Hess}^{\tilde g}u^m + \nabla^{\tilde g} X)(\xi,\xi) \\
			&= \xi^i\left(\partial_i\partial_j u^m - \tilde \Gamma_{ij}^k \partial_k u^m + \langle\nabla_{\partial_i}^{\tilde g} X,\partial_j\rangle_{\tilde  g}\right)\xi^j \\
			&= \xi^i\left(u^m_{ij} - \tilde \Gamma_{ij}^k u^m_k - \partial_i X^k \tilde g_{kj} + X^k \tilde \Gamma_{ik}^l \tilde g_{lj}\right)\xi^j \\
			&=  \underbrace{\int \eta^m(x-y)u_{ij}(y)\xi^i\xi^j dy 
				- \tilde \Gamma_{ij}^k(x) \int \eta^m(x-y)u_k(y)\xi^i\xi^j dy}_{(A)} \\
			&\quad + \underbrace{\xi^i\xi^j\left(-\partial_i X^k(x)\tilde  g_{kj}(x) + X^k(x) \tilde \Gamma_{ik}^l(x) \tilde g_{lj}(x)\right)}_{(B)}.
		\end{align*}
		
		We can rewrite the sum as
		\[
		\begin{aligned}
			(A)+(B) &= \int \eta^m(x-y)\left[u_{ij}(y) - \tilde \Gamma_{ij}^k(y)u_k(y) - \partial_i X^k(y)\tilde g_{kj}(y) + X^k(y)\tilde \Gamma_{ik}^l(y)\tilde g_{lj}(y)\right]\xi^i\xi^j dy \quad (1) \\
			&\quad + \int \eta^m(x-y)\left[\tilde \Gamma_{ij}^k(y) - \tilde \Gamma_{ij}^k(x)\right]u_k(y)\xi^i\xi^j dy\quad (2) \\
			&\quad + \left[\left(\eta^m \ast \left(\partial_i X^k \tilde g_{kj} - X^k \tilde \Gamma_{ik}^l \tilde g_{lj}\right)\right)(x) - \left(\partial_i X^k \tilde g_{kj} - X^k \tilde \Gamma_{ik}^l \tilde g_{lj}\right)(x)\right]\xi^i\xi^j. \quad (3)
		\end{aligned}
		\]
		
		Since
		\[
		\begin{aligned}
			\int \eta^m(x - y) (\xi^i \tilde g_{ij}(y) \tilde{v}^j(y))^2 dy &= \int (\eta^m)^{\frac{1}{2}\times 2}(x - y) dy \int (\eta^m)^{\frac{1}{2}\times 2}(x - y) (\xi^i \tilde g_{ij}(y) \tilde{v}^j(y))^2 dy \\
			&\geq \left( \int \eta^m(x - y) \xi^i \tilde g_{ij}(y) \tilde{v}^j(y) dy \right)^2 \\
			&= \left( \int \eta^m(x - y) \xi^i \tilde{v}_i(y) dy \right)^2 \\
			&= \left( \int \eta^m(x - y) \xi^i (u_i(y) + \omega_i(y)) dy \right)^2 \\
			&= \left\{ \xi^i (\eta^m * u_i)(x) + \xi^i \omega_i(x) + o\left( \frac{1}{m} \right) \right\}^2 \\
			&= \left\{ \xi^i  u_i^m(x) + \xi^i \omega_i(x) + o\left( \frac{1}{m} \right) \right\}^2 \\
			&= \langle \xi, \nabla^{\tilde g} u^m + X \rangle^2_{\tilde g} + o\left( \frac{1}{m} \right),
		\end{aligned}
		\]
		we can obtain
		\[
		\begin{aligned}
			(1) &\leq \kappa \int \eta^m(x - y) \left\{ \xi^i \tilde g_{ij}(y) \xi^j - \frac{1}{2} (\xi^i \tilde g_{ij}(y) \tilde{v}^j(y))^2 \right\} dy \\
			&\leq \kappa \left\{ \int \eta^m(x - y) \xi^i \tilde g_{ij}(y) \xi^j dy - \frac{1}{2}\langle \xi, \nabla^{\tilde g} u^m + X \rangle^2_{\tilde g} \right\} + o\left( \frac{1}{m} \right) \\
			&= \kappa \left\{ \langle \xi, \xi \rangle_{\tilde g} - \frac{1}{2} \langle \xi, \nabla^{\tilde g} u^m + X \rangle^2_{\tilde g} \right\} + o\left( \frac{1}{m} \right).
		\end{aligned}
		\]
		Note that
		\[
		\begin{aligned}
			(2) &= \int (\eta^m)^{\frac{1}{2}} \left( \tilde \Gamma_{ij}^k(y) - \tilde \Gamma_{ij}^k(x) \right) (\eta^m)^{\frac{1}{2}} u_k(y) \xi^i \xi^j dy \\
			&\leq \sum_{i,j,k} \left( \int \eta^m(x - y) \left( \tilde \Gamma_{ij}^k(y) - \tilde \Gamma_{ij}^k(x) \right)^2dy \right)^{\frac{1}{2}} \left( \int \eta^m(x - y) \left( u_k(y) \xi^i \xi^j \right)^2 dy \right)^{\frac{1}{2}} \\
			&= o\left( \frac{1}{m} \right)
		\end{aligned}
		\]
		and $ (3)= o\left( \frac{1}{m} \right)$.
		Therefore,
		\[
		\begin{aligned}
			(\operatorname{Hess}^{\tilde g} u^m + \nabla^{\tilde g} X)(\xi, \xi) &= \xi^i \left( u_{ij}^m - \tilde \Gamma_{ij}^k u_k^m - \partial_i X^k \tilde g_{kj} + X^k \tilde \Gamma_{ik}^l \tilde g_{lj} \right) \xi^j \\
			&\leq (1) + (2) + (3) \\
			&\leq \kappa \left\{ \langle \xi, \xi \rangle_{\tilde g} - \frac{1}{2} \langle \xi, \nabla^{\tilde g} u^m + X \rangle^2_{\tilde g} \right\} + o\left( \frac{1}{m} \right).
		\end{aligned}
		\]
		
		Now, let
		\[
		\psi_m(t) = \frac{1}{2} \langle \nabla^{\tilde g} u^m + X, \nabla^{\tilde g} u^m + X \rangle_{\tilde g} (\gamma_m(t)) - 1=\frac{1}{2}\|Du_m(\gamma_m(t))+\omega \|_{\tilde g^*}^2-1
		\]
		with
		\[
		\left\{ \begin{array}{l} \dot{\gamma}(t) = \nabla^{\tilde g} u^m(\gamma_m(t)) + X(\gamma_m(t)), \, t \in [0, T], \\ \gamma(0) = x_0. \end{array} \right.
		\]
		Then,
		\[
		\begin{aligned}
			\psi_m'(t) &= (\operatorname{Hess}^{\tilde g} u^m + \nabla^{\tilde g} X)(\nabla^{\tilde g} u^m + X, \nabla^{\tilde g} u^m + X)(\gamma_m(t)) \\
			&\leq \kappa \left\{ \langle \nabla^{\tilde g} u^m + X, \nabla^{\tilde g} u^m + X \rangle_{\tilde g} - \frac{1}{2} \langle \nabla^{\tilde g} u^m + X, \nabla^{\tilde g} u^m + X \rangle^2_{\tilde g} \right\} + o\left( \frac{1}{m} \right) \\
			&= \kappa \langle \nabla^{\tilde g} u^m + X, \nabla^{\tilde g} u^m + X \rangle_{\tilde g} \left( - \psi_m(t) \right) + o\left( \frac{1}{m} \right) \\
			&\leq -\tilde{C} \psi_m(t) + o\left( \frac{1}{m} \right)
		\end{aligned}
		\]
		for some constant \( \tilde{C} > 0 \). Therefore,
		\[
		\psi_m(T) \leq e^{-\tilde{C} T} \psi_m(0) + \int_0^T e^{\tilde{C}(s - T)} o\left( \frac{1}{m} \right) ds.
		\]
		Since $\psi_m(0)  =\frac{1}{2}\|Du_m(x_0)+\omega \|_{\tilde g^*}^2-1 \to \frac{1}{2}\|p_0+\omega \|_{\tilde g^*}-1=b<0 $, we have $\liminf\limits_{m\to \infty} \psi_m(t) \leq e^{-\tilde CT}b < 0$.
		
		Set $\bar H(x,p):=\tilde H(x,p)-1=\frac{1}{2}\|p+\omega \|_{\tilde g^*}^2-1$. According to the Arzel\`a--Ascoli theorem, we may assume, up to a subsequence, that the curves \(\gamma_m(t)\) converge uniformly to a Lipschitz curve. From Lemma \ref{uniform semiconcave}, we deduce that this limiting curve is a solution to \eqref{differential inclusion}. Together with Proposition \ref{pro:uniqueness}, it follows that the entire sequence \(\gamma_m(t)\) converges uniformly to the Lipschitz curve \(\gamma(t)\) for \(t \in [0,T]\). Moreover, by Lemma \ref{uniform semiconcave}, we have that 
		$$  0>\liminf\limits_{m\to \infty} \psi_m(t)=\liminf\limits_{m\to \infty} \bar H(\gamma_m(t),Du^m(\gamma_m(t)))\geq \bar H(\gamma(t),\bar p^\#_u(\gamma(t))),$$
		with $\bar p^\#_u:=\arg\min\{\bar H(x,p)\mid p\in D^+u(x)\}$
		
		We now claim that $\bar H(\gamma(t),\bar p^\#_u(\gamma(t)))<0$ if and only if $\gamma(t)\in \SING(u)$. Indeed, we have $\gamma(t)\in \SING(u)$ if and only if $H(\gamma(t), p^\#_u(\gamma(t)))<c[L]$ and $H(\gamma(t), p^\#_u(\gamma(t)))<c[L]$ if and only if 
		$\bar H(\gamma(t),\bar p^\#_u(\gamma(t)))<0$. Thus, the proof follows from the claim.
	\end{proof}

	\begin{Cor}\label{cor:cut to sing}
		Under the assumptions in Theorem \ref{thm:singular_persistence}, for any initial point $x_0 \in W \cap \text{\rm Cut}\,(u)$, the generalized gradient flow $\gamma(t)$ with $\gamma(0) = x_0$ satisfies:
		\[
		\gamma(t) \in \SING(u) \quad \text{for all } t >0 \text{ where the flow is defined.}
		\]
	\end{Cor}
	
	\begin{proof}
		This follows directly from Theorem 5.7 in \cite{CCHW2024} combined with the uniqueness result in Proposition \ref{pro:uniqueness}. By Theorem 5.7 in \cite{CCHW2024}, the set $\{ t \in [0,+\infty) : \gamma(t) \in \mathrm{Sing}(u) \}$ has interior dense in $[0,+\infty)$. Thus for any $t>0$, there exists $t_1 \in (0,t)$ such that $\gamma(t_1) \in \mathrm{Sing}(u)$. The uniqueness property established in Proposition \ref{pro:uniqueness} then implies $\gamma(t) \in \mathrm{Sing}(u)$.
	\end{proof}

	\bibliographystyle{plain}
	\bibliography{mybib}

\begin{thebibliography}{10}

\bibitem{Albano_Cannarsa2002}
Paolo Albano and Piermarco Cannarsa.
\newblock Propagation of singularities for solutions of nonlinear first order partial differential equations.
\newblock {\em Arch. Ration. Mech. Anal.}, 162(1):1--23, 2002.

\bibitem{ACCM2025}
Paolo Albano, Piermarco Cannarsa, Wei Cheng, and Cristian Mendico.
\newblock Long time behaviour of generalised gradient flows of solutions to {H}amilton-{J}acobi equations.
\newblock preprint, 2025.

\bibitem{ACNS2013}
Paolo Albano, Piermarco Cannarsa, Khai~Tien Nguyen, and Carlo Sinestrari.
\newblock Singular gradient flow of the distance function and homotopy equivalence.
\newblock {\em Math. Ann.}, 356(1):23--43, 2013.

\bibitem{Ambrosio_GigliNicola_Savare_book2008}
Luigi Ambrosio, Nicola Gigli, and Giuseppe Savar\'{e}.
\newblock {\em Gradient flows in metric spaces and in the space of probability measures}.
\newblock Lectures in Mathematics ETH Z\"{u}rich. Birkh\"{a}user Verlag, Basel, second edition, 2008.

\bibitem{Bangert1979}
Victor Bangert.
\newblock Analytische {E}igenschaften konvexer {F}unktionen auf {R}iemannschen {M}annigfaltigkeiten.
\newblock {\em J. Reine Angew. Math.}, 307(308):309--324, 1979.

\bibitem{Cannarsa_Cheng3}
Piermarco Cannarsa and Wei Cheng.
\newblock Generalized characteristics and {L}ax-{O}leinik operators: global theory.
\newblock {\em Calc. Var. Partial Differential Equations}, 56(5):Art. 125, 31, 2017.

\bibitem{Cannarsa_Cheng_Fathi2017}
Piermarco Cannarsa, Wei Cheng, and Albert Fathi.
\newblock On the topology of the set of singularities of a solution to the {H}amilton-{J}acobi equation.
\newblock {\em C. R. Math. Acad. Sci. Paris}, 355(2):176--180, 2017.

\bibitem{Cannarsa_Cheng_Fathi2021}
Piermarco Cannarsa, Wei Cheng, and Albert Fathi.
\newblock Singularities of solutions of time dependent {H}amilton-{J}acobi equations. {A}pplications to {R}iemannian geometry.
\newblock {\em Publ. Math. Inst. Hautes \'{E}tudes Sci.}, 133(1):327--366, 2021.

\bibitem{CCHW2024}
Piermarco Cannarsa, Wei Cheng, Jiahui Hong, and Kaizhi Wang.
\newblock Variational construction of singular characteristics and propagation of singularities.
\newblock preprint, arXiv 2409.00961, 2024.

\bibitem{Cannarsa_Mazzola_Sinestrari2015}
Piermarco Cannarsa, Marco Mazzola, and Carlo Sinestrari.
\newblock Global propagation of singularities for time dependent {H}amilton-{J}acobi equations.
\newblock {\em Discrete Contin. Dyn. Syst.}, 35(9):4225--4239, 2015.

\bibitem{Cannarsa_Sinestrari_book}
Piermarco Cannarsa and Carlo Sinestrari.
\newblock {\em Semiconcave functions, {H}amilton-{J}acobi equations, and optimal control}, volume~58 of {\em Progress in Nonlinear Differential Equations and their Applications}.
\newblock Birkh{\"a}user Boston, Inc., Boston, MA, 2004.

\bibitem{Cannarsa_Yu2009}
Piermarco Cannarsa and Yifeng Yu.
\newblock Singular dynamics for semiconcave functions.
\newblock {\em J. Eur. Math. Soc. (JEMS)}, 11(5):999--1024, 2009.

\bibitem{Fathi_book}
Albert Fathi.
\newblock {W}eak {KAM} theorem in {L}agrangian dynamics.
\newblock Cambridge University Press, Cambridge (to appear).

\bibitem{Greene_Wu1972}
R.~E. Greene and H.~Wu.
\newblock On the subharmonicity and plurisubharmonicity of geodesically convex functions.
\newblock {\em Indiana Univ. Math. J.}, 22:641--653, 1972/73.

\bibitem{Greene_Wu1974}
R.~E. Greene and H.~Wu.
\newblock Integrals of subharmonic functions on manifolds of nonnegative curvature.
\newblock {\em Invent. Math.}, 27:265--298, 1974.

\bibitem{Khanin_Sobolevski2016}
Konstantin Khanin and Andrei Sobolevski.
\newblock On dynamics of {L}agrangian trajectories for {H}amilton-{J}acobi equations.
\newblock {\em Arch. Ration. Mech. Anal.}, 219(2):861--885, 2016.

\bibitem{Ohta2009}
Shin-ichi Ohta.
\newblock Uniform convexity and smoothness, and their applications in {F}insler geometry.
\newblock {\em Math. Ann.}, 343(3):669--699, 2009.

\bibitem{Villani_book2009}
C\'{e}dric Villani.
\newblock {\em Optimal transport: old and new}, volume 338 of {\em Grundlehren der Mathematischen Wissenschaften}.
\newblock Springer-Verlag, Berlin, 2009.

\bibitem{Yu2006}
Yifeng Yu.
\newblock A simple proof of the propagation of singularities for solutions of {H}amilton-{J}acobi equations.
\newblock {\em Ann. Sc. Norm. Super. Pisa Cl. Sci. (5)}, 5(4):439--444, 2006.

\end{thebibliography}
\end{document}